\documentclass[3p,10pt,nopreprintline]{elsarticle}

\usepackage{amsmath}
\usepackage{amsthm}
\usepackage{amssymb}
\usepackage{graphicx}
\usepackage{epsfig}
\usepackage[normal]{subfigure}
\usepackage{float}
\usepackage{paralist}
\usepackage{caption2}
\usepackage{hyperref}
\usepackage{aliascnt}
\usepackage[T1]{fontenc}
\usepackage{natbib}
\usepackage{CJK}
\usepackage{color}
\usepackage[table]{xcolor}
\usepackage{colortbl,booktabs}

\newtheoremstyle{theorem}{6pt}{6pt}{\itshape}{}{\bfseries}{.}{.5em}{}
\newtheoremstyle{definition}{6pt}{6pt}{\upshape}{}{\bfseries}{.}{.5em}{}

\theoremstyle{theorem}
\newtheorem{theorem}{Theorem}[section]

\newaliascnt{corollary}{theorem}

\aliascntresetthe{corollary}

\newaliascnt{lemma}{theorem}
\newtheorem{lemma}[lemma]{Lemma}
\aliascntresetthe{lemma}

\newaliascnt{sublemma}{theorem}

\aliascntresetthe{sublemma}

\theoremstyle{definition}

\newtheorem{remark}{Remark}[section]

\newaliascnt{proposition}{theorem}
\newtheorem{proposition}[proposition]{Proposition}
\aliascntresetthe{proposition}

\newcommand{\dif}{{\mathrm d}}

\newcommand{\bn}{\begin{eqnarray}}
\newcommand{\en}{\end{eqnarray}}
\newcommand{\bnn}{\begin{eqnarray*}}
\newcommand{\enn}{\end{eqnarray*}}
\renewcommand{\div}{ {\rm div }  }

\newcommand{\al}{\alpha}

\newcommand{\calD}{\mathcal{D}}

\newcommand{\calE}{\mathcal{E}}

\newcommand{\D}{\nabla}

\allowdisplaybreaks

\numberwithin{equation}{section}

\begin{document}

\begin{frontmatter}

\title{Global well-posedness and exponential decay of strong solution for the \\ three-dimensional inhomogeneous incompressible micropolar equations with \\ density-dependent transport coefficients and large initial data}

\author[label1]{Peng Lu}
\address[label1]{Department of Mathematics, Zhejiang Sci-Tech University, Hangzhou, Zhejiang 310018, P.R.China;}
\ead{plu25@zstu.edu.cn}

\author[label2]{Yuanyuan Qiao\corref{cor2}}
\address[label2]{School of Mathematics and Information Science, Henan Polytechnic University, Jiaozuo, Henan 454000, P.R.China;}
\ead{yyqiao@hpu.edu.cn}

\cortext[cor2]{Corresponding author. }

\begin{abstract}
In this paper, we consider the Dirichlet problem of three-dimensional inhomogeneous incompressible micropolar equations with density-dependent viscosity. Under the assumption that the coefficients are power functions of the density, we establish the global existence of strong solutions as long as the initial density is linear equivalent to a large constant state. There is no restriction on the size of initial velocity and micro-rotational velocity. As a by-product, we prove the exponential decay for the solution.
\end{abstract}

\end{frontmatter}

\section{Introduction}

Let $\Omega\subset\mathbb{R}^3$ be a bounded domain with smooth boundary. In this paper, we are concerned with an initial boundary-value problem of three-dimensional (3D for short) inhomogeneous micropolar fluid equations with density-dependent viscosity in $\Omega\times\mathbb{R}_+$:
\begin{equation}\label{ins}
\left\{\begin{array}{l}
\rho_t+\mathrm{div}(\rho u)=0, \\
(\rho u)_t+\mathrm{div}(\rho u\otimes u)+\nabla P-2\mathrm{div}((\mu(\rho)+\xi(\rho) )\calD(u))=2\xi(\rho)\D\times w, \\
(\rho w)_t+\mathrm{div}(\rho u\otimes w)+4\xi(\rho)w-\div\left(2\eta(\rho)\calD(w)\right)-\D(\lambda(\rho)\div w)=2\xi(\rho)\D\times u, \\
\div u=0,
\end{array}\right.
\end{equation}
where $t\geq 0$ and $x=(x_1,x_2,x_3)\in\Omega$ are time and space variables, respectively. $\rho$, $u=(u^1, u^2, u^3)$, $w=(w^1, w^2, w^3)$ and $P$ represent the density, velocity, micro-rotational velocity, and the pressure of the fluid, respectively.
\begin{equation}
\calD(u)=\frac12 \left(\nabla u+(\nabla u)^\top\right), \quad \calD(w)=\frac12 \left(\nabla w+(\nabla w)^\top\right),
\end{equation}
is the deformation tensor. $\mu(\rho)$ stands for the viscosity and is a function of the density satisfying
\begin{equation}\label{vis-d}
\mu(\rho)=\bar\mu\rho^\alpha,\ \bar\mu>0,\ \alpha\geq 0. 
\end{equation}
The coefficients $\xi(\rho)$, $\eta(\rho)$, $\lambda(\rho)$ are assumed to satisfy
\begin{equation}\label{other-coeffi}
\eta(\rho)=\bar\eta\rho^\al, ~ \lambda(\rho)=\bar\lambda\rho^\al, ~ \xi(\rho)=\bar\xi\rho^\beta, \quad \bar\xi>0, \quad \bar\eta>0, \quad 2\bar\eta+3\bar\lambda\geq 0, ~ \beta\geq 0.
\end{equation}

In this paper, we study the initial boundary value problem to the system \eqref{ins}--\eqref{other-coeffi} with the following initial data and Dirichlet boundary condition:
\begin{equation}\label{bc}
\begin{aligned}
& (\rho, u, w)(x,0)=(\rho_0(x), u_0(x), w_0(x)), & \text{in} ~ & \Omega, \\
& u(x,t)=0, ~ w(x,t)=0, & \mathrm{on} ~ & \partial\Omega\times [0,T].
\end{aligned}
\end{equation}

The micropolar equations were first proposed by Eringen \cite{MR204005} in 1966, which can describe many phenomena that appear in a large number of complex fluids such as suspensions, animal blood, and liquid crystals. For more background on micropolar fluids, we refer to \cite{MR4157866, MR3944649, MR1711268} and the references therein. Because of their physical applications and mathematical significance, the well-posedness problem of the micropolar equations has attracted considerable attention in recent years. In 2019, Zhang and Zhu \cite{MR3950716} considered the 3D Cauchy problem with constant viscosity, and proved the global well-posedness of strong and classical solution under the condition that $\bar\mu$ is large enough or the initial data is sufficiently small. In 2020, Ye \cite{MR4026901} studied the large time behavior of the global solution obtained in \cite{MR3950716}. In 2023, Li \cite{MR4533917} established a weak Serrin-type blowup criterion for the 3D Cauchy problem, based on which the author obtained a unique global strong solution. When the viscosity depends on the density, Zhong \cite{MR4514901} studied the 3D initial-boundary-value problem, established the global existence and uniqueness of strong solutions when the initial energy is sufficiently small. Moreover, the velocity and the micro-rotational velocity converge exponentially to zero in $H^2$ as time goes to infinity. Zhou and Tang \cite{MR4672993} proved the global well-posedness of strong solution to the 3D Cauchy problem, provided that the initial mass is sufficiently small. Moreover, the gradients of velocity and micro-rotational velocity converge exponentially to zero in $H^1$ as time goes to infinity. For more related results in two-dimensional case, we refer to \cite{MR4275463, MR4175101, MR4227470, MR4470525}.

If there is no microstructure $(\xi = 0 ~ \text{and} ~ w = 0)$, the system \eqref{ins} reduces to inhomogeneous incompressible Navier–Stokes equations. When the viscosity is constant, Choe and Kim \cite{MR1986066} proved the well-posedness of local strong solutions to the Cauchy problem or the initial boundary value problem. In 2008, Germain \cite{MR2438424} proved weak-strong uniqueness to the Cauchy problem. In 2013, Craig, Huang and Wang \cite{MR3127017} established the global well-posedness of strong solutions for initial data with small $\dot{H}^{\frac{1}{2}}$-norm. In 2016, Chen, Zhang and Zhao \cite{MR3487274} proved the global well-posedness of solution under the condition that the initial density is bounded and bounded away from zero, and the initial velocity is small enough in $H^s$ for some $s>1/2$. In 2021, Guo, Wang and Xie \cite{MR4347323} established global strong axisymmetric solutions in the exterior of a cylinder subject to the Dirichlet boundary conditions, where the vacuum is allowed. When the viscosity depends on the density, Huang and Wang \cite{MR3200422} studied the 2D initial boundary value problem in bounded domains, and proved that strong solutions exist globally when $\|\D\mu(\rho_0)\|_{L^p}$ is sufficiently small, even in the presence of vacuum. In 2015, Liang \cite{MR3306351} established the local well-posedness of strong solution for the 2D Cauchy problem, Huang and Wang \cite{MR3345862} and Zhang \cite{MR3349417} proved that 3D initial boundary value problem admits unique strong solution provided that $\|\D u_0\|_{L^2}$ is sufficiently small. For more related results in critical Besov spaces, we refer to \cite{MR4334732, MR2663713, MR3073212, MR4759529, MR3412410, MR3369261, MR3619877, MR2957705, MR4580531, MR3056619, MR2889168, MR4093854, MR3169743, MR4568414, MR4454275, MR3582198, MR3721433, MR4056004} and the references therein. 

Recently, Huang, Li and Zhang \cite{Huang-Li-Zhang-arxiv} studied the Dirichlet problem of 3D inhomogeneous Navier-Stokes equations with density-dependent viscosity, and proved that the system admits a unique global strong solution as long as the initial density are sufficiently large. This is the first result concerning the existence of large strong solution for the inhomogeneous Navier-Stokes equations in three dimensions. Motivated by \cite{Huang-Li-Zhang-arxiv}, we consider the global well-posedness problem of inhomogeneous incompressible micropolar equations with density-dependent viscosity and large initial data.

Before stating the main results, we explain the notation and conventions used throughout this paper. For given function $f(x)$ or $f(x,t)$, we denote
\begin{equation*}
\int f \dif x=\int_{\Omega} f\dif x,
\end{equation*}
For a positive integer $k$ and $p\geq1$, we denote the standard Lesbegue and Sobolev spaces as follows:
\begin{equation*}
\begin{gathered}
\|f\|_{L^p}=\|f\|_{L^p(\Omega)},\ \|f\|_{W^{k,p}}=\|f\|_{W^{k,p}(\Omega)},\ \|f\|_{H^k}=\|f\|_{W^{k,2}(\Omega)},\\
C^\infty_{0,\sigma}=\{f\in C^\infty_0(\Omega):\mathrm{div}f=0\}, \ H_0^1=\overline{C^\infty_0}, \\
H_{0,\sigma}^1=\overline{C^\infty_{0,\sigma}},\ \mathrm{closure}\ \mathrm{in}\ \mathrm{the}\ \mathrm{norm}\ \mathrm{of}\  H^1.    
\end{gathered}
\end{equation*}

Our main result can be stated as follows.
\begin{theorem}\label{global} 
Let $\Omega$ be a bounded smooth domain in $\mathbb{R}^3$. Assume that
\begin{equation}\label{vis-c}
\alpha>1, \quad 0<\beta\leq\frac{\alpha+1}{2}.
\end{equation}
Given constants $\bar{\rho}>1$ and $C_0>1$, suppose that the initial data $(\rho_0,u_0, w_0)$ satisfies 
\begin{equation}\label{ia}
\bar \rho \le \rho_0\le C_0 \bar \rho,\quad \rho_0 \in W^{1,q},\ 3<q<6, \quad  u_0 \in H_{0,\sigma}^1 \cap H^2, \quad  w_0 \in H_0^1 \cap H^2.
\end{equation}
Then there exists a positive constant $M$ depending only on $C_0,\bar\mu, \bar\xi, \bar\eta, \bar\lambda, \alpha$, $\|\nabla\rho_0\|_{L^q}, \| u_0\|_{H^2}, \| w_0\|_{H^2}$ and $\Omega$ such that if
\begin{equation}\label{ini}
\bar\rho \ge M(\Omega,C_0,\bar\mu,\bar\xi,\bar\eta,\bar\lambda,\alpha,\|\nabla\rho_0\|_{L^q}, \| u_0\|_{H^2}, \| w_0\|_{H^2}),
\end{equation}
then the problem \eqref{ins}--\eqref{bc} admits a unique global strong solution $(\rho,u)$ in $\Omega\times(0,\infty)$ satisfying 
\begin{equation}
\left\{ \begin{array}{l}
\rho\in C([0,\infty);W^{1,q}),\quad \rho_t \in C([0,\infty);L^q), \\
\nabla u, P\in C([0,\infty);H^1)\cap L^2((0,\infty);W^{1,q}),\quad \nabla w\in C([0,\infty);H^1), \\
(\sqrt{\rho}u_t, \sqrt{\rho}w_t)\in L^\infty(0,\infty;L^2),\quad (u_t, w_t)\in L^2(0,\infty;H^1_0).
\end{array} \right.
\end{equation}

Moreover, there exist a positive constant $\kappa(\bar{\rho})$ depending on $\bar{\rho}$, $C_0$, $\bar{\mu}$, $\bar{\xi}$, $\bar{\eta}$, $\bar{\lambda}$, $\alpha$, $\|\nabla\rho_0\|_{L^q}$, $\|u_0\|_{H^2}$, $\| w_0\|_{H^2}$ and $\Omega$ such that
\begin{equation}
\|u\|_{H^1}^2+\|w\|_{H^1}^2\leq Ce^{\kappa(\bar{\rho})t}.
\end{equation}
\end{theorem}

\begin{remark}
As was pointed out by \cite{Huang-Li-Zhang-arxiv}, Theorem \ref{global} implies that the flow is globally stable when the initial Reynolds number is sufficiently small. 
\end{remark}
\begin{remark}
Theorem \ref{global} can also be applied to the periodic case. Moreover, after slight modifications, we can also treat the case $\alpha>1$ and $\beta=0$.
\end{remark}

Let's briefly sketch the proof. To extend the local solution to a global one, we need to establish uniform a priori estimates on solutions in suitable higher-order norms, where some basic ideas are borrowed from \cite{Huang-Li-Zhang-arxiv}. However, due to the strong coupling between the velocity $u$ and the micro-rotational velocity $w$, the proof of Theorem \ref{global} is much more complicated. Due to the linear terms $\xi(\rho)\D\times w$ and $\xi(\rho)\D\times u$ on the right hand side of \eqref{ins}$_2$ and \eqref{ins}$_3$, there are extra linear terms in the high-order estimates of $u$ and $w$ (see Lemma \ref{lem:D2u}), which makes it impossible for us to follow the idea of \cite{Huang-Li-Zhang-arxiv} to prove the time-weighted estimates \eqref{tdu}--\eqref{t2du}. Inspired by \cite{MR4514901}, we find that $e^{\kappa t}(\|\D u\|_{L^2}^2+\|\D w\|_{L^2}^2)$ is integrable on $[0,T]$ for some $\kappa>0$ depending on $\bar{\rho}$, with an order of $\bar{\rho}^{1-\alpha}$ (see \eqref{basic-est.2}). With the help of \eqref{basic-est.2}, 
 we are able to prove the desired time-weighted estimates, which is essential in establishing the uniform bound of $\|\nabla u\|_{L^1_tL^\infty_x}$.

The rest of the paper is organized as follows. In Section \ref{prelim}, we collect some elementary facts and inequalities which will be needed in later analysis. Section \ref{apri} is devoted to the proof of Theorem \ref{global}.

\section{Preliminaries}\label{prelim}
First, the following local existence theory, where the initial density is strictly away from vacuum, can be shown by similar arguments as in Cho and Kim \cite{MR2094425}:
\begin{lemma}\label{local}
Assume that the initial data $(\rho_0, u_0, w_0)$ satisfies the regularity condition \eqref{ia}. Then there exists a small time $T_0$ and a unique strong solution $(\rho, u, P)$ to the initial boundary value problem \eqref{ins}--\eqref{bc} such that
\begin{equation}\label{l-r}
\left\{ \begin{array}{l}
\rho\in C([0,T_0];W^{1,q}),\quad \rho_t \in C([0,T_0];L^q), \\
\nabla u, P\in C([0,T_0];H^1)\cap L^2(0,T_0;W^{1,q}),\quad \nabla w\in C([0,T_0];H^1), \\
(\sqrt{\rho}u_t, \sqrt{\rho}w_t)\in L^\infty(0,T_0;L^2),\quad (u_t, w_t)\in L^2(0,T_0;H^1_0).
\end{array} \right.
\end{equation}
Furthermore, if $T^\ast$ is the maximal existence time of the local strong solution $(\rho, u)$, then either
$T^\ast=\infty$ or
\begin{equation}\label{blow-up}
    \sup_{0\leq t\leq T^\ast}\left(\|\nabla\rho\|_{L^q}+\|\nabla u\|_{L^2}+\|\nabla w\|_{L^2}\right)=\infty.
\end{equation}
\end{lemma}
In this paper, we will employ Bovosgii's theory which can be found in \cite{MR3289443}.
\begin{lemma}\label{bovosgii}
Let $\Omega$ be a bounded domain with Lipschitz boundary, $1 < p < \infty$, and $b(x) \in L^p(\Omega)$ with $\int_\Omega b(x) \dif x=0$. Then there exists $v(x) \in W_0^{1,p}(\Omega)$ satisfying:
\begin{equation*}
\mathrm{div} v=b, \qquad\text{in} ~ \Omega,
\end{equation*}
and
\begin{equation}
\|\nabla v\|_{L^p(\Omega)}\leq C(p)\|b\|_{L^p(\Omega)}.
\end{equation}
\end{lemma}
Also, the well-known Gagliardo-Nirenberg inequality \cite{MR241822} will be frequently used in this paper.
\begin{lemma}\label{G-N}
Assume that $\Omega$ is a bounded Lipschitz domain in $\mathbb{R}^3$. Let 1 $\leq q  \leq\infty$ be a positive extended real quantity. Let j and m  be non-negative integers such that $ j < m$. Furthermore, let $1 \leq r \leq \infty$ be a positive extended real quantity,  $p \geq 1$  be real and  $\theta \in [0,1]$  such that the relations
\begin{equation}
    \dfrac{1}{p} = \dfrac{j}{n} + \theta \left( \dfrac{1}{r} - \dfrac{m}{n} \right) + \dfrac{1-\theta}{q}, \qquad \dfrac jm \leq \theta \leq 1
\end{equation}
hold. Then, 
\begin{equation}
    \|\nabla^j u\|_{L^p(\Omega)} \leq C\|\nabla^m u\|_{L^r(\Omega)}^\theta\|u\|_{L^q(\Omega)}^{1-\theta} + C_1\|u\|_{L^q(\Omega)}
\end{equation}
where $u \in L^q(\Omega)$  such that  $\nabla^m u \in L^r(\Omega)$. Moreover, if $q>1$ and $r>3$,
\begin{equation}
\|u\|_{C(\bar{\Omega})}\leq C\|u\|^{q(r-3)/(3r+q(r-3))}_{L^q}\|\nabla u\|^{3r/(3r+q(r-3))}_{L^r}+C_2\|u\|_{L^q}.
\end{equation}
where $u \in L^q(\Omega)$  such that  $\nabla u \in L^r(\Omega)$.
In any case, the constant $C > 0$  depends on the parameters $j,\,m,\,n,\,q,\,r,\,\theta$, on the domain $\Omega$, but not on $u$. In addition, if $u\cdot n|_{\partial\Omega}=0$ or $\int_\Omega u\dif x=0$, we can choose $C_1=C_2=0$.
\end{lemma}

High-order a priori estimates rely on the following regularity results for density-dependent Stokes equations. The proof can be found in \cite{Huang-Li-Zhang-arxiv}.
\begin{lemma}\label{stokes-e}
Assume that $\rho\in W^{1,q}, 3<q<6$, and $\bar \rho \le \rho \le C_0 \bar \rho$.
Let $(u, P) \in H_{0,\sigma}^1\times L^2$ be the unique weak solution to the boundary value problem
\begin{equation}\label{stokes}
	\left\{ \begin{array}{l}
		-\mathrm{div}(2\bar\mu\rho^\alpha \calD(u))+\nabla P=F, \\
		\mathrm{div} u=0, \\
        \int \frac{P}{\rho^\alpha} \dif x=0.
	\end{array} \right.
\end{equation}
Then we have the following regularity results:

(1) If $F\in L^2$, then $(u, P)\in H^2\times H^1$ and
\begin{equation}\label{stokes-2}
\| u\|_{H^2}+\bigg\|\frac{P}{\rho^\alpha}\bigg\|_{H^1}
\leq C(\bar\rho^{-\alpha}+\bar\rho^{-\alpha-\frac{q}{q-3}}\|\nabla \rho \|_{L^q}^\frac{q}{q-3})\|F\|_{L^2};
\end{equation}

(2) If $F\in L^q$ for some $q\in (3,6)$ then $(u, P)\in W^{2,q}\times W^{1,q}$ and
\begin{equation}\label{stokes-q}
\| u\|_{W^{2,q}}+\bigg\|\frac{P}{\rho^\alpha}\bigg\|_{W^{1,q}}
\leq  C(\bar\rho^{-\alpha}+\bar\rho^{-\alpha-\frac{5q-6}{2(q-3)}}\|\nabla \rho \|_{L^q}^\frac{5q-6}{2(q-3)})\|F\|_{L^q},
\end{equation}
where the constant C in \eqref{stokes-2} and \eqref{stokes-q} depends on $\Omega, q$.
\end{lemma}

\section{A priori estimates}\label{apri}

For any fixed time $T>0$, $(\rho,u,w,P)$ is the unique local strong solution to \eqref{ins}--\eqref{bc} on $\Omega\times (0,T]$ with initial data $(\rho_0,u_0,w_0)$ satisfying \eqref{ia}, which is guaranteed by Lemma \ref{local}. Define
\begin{gather}\label{As1}
	\mathcal{E}_\rho(T) := \sup_{t\in[0,T] }\|\nabla \rho \|_{L^q},\\
	\mathcal{E}_w(T) := \bar\rho^{\alpha}\sup_{t\in[0,T] }\|\nabla w \|_{L^2}^2
	+  \int_0^{T}  \|\sqrt{\rho} w_t \|_{L^2}^2\dif t,\\
	\mathcal{E}_u(T) := \bar\rho^{\alpha}\sup_{t\in[0,T] }\|\nabla u \|_{L^2}^2
	+  \int_0^{T}  \|\sqrt{\rho} u_t \|_{L^2}^2\dif t.
\end{gather}

We have the following key proposition.
\begin{proposition}\label{pr}
Under the conditions of Theorem \ref{global}, there exists positive constants $K$ and $M$ depending on $C_0,\bar\mu, \bar\xi, \bar\eta, \bar\lambda, \alpha, \beta$, $\|\nabla\rho_0\|_{L^q}, \| u_0\|_{H^2}, \| w_0\|_{H^2}$ and $\Omega$ such that if $(\rho,u,w,P)$ is a smooth solution to the problem (\ref{ins})--(\ref{bc}) on $\Omega\times (0,T]$ satisfying
	\begin{equation}\label{a1}
		\mathcal{E}_\rho(T) \le 3\mathcal{E}_\rho(0), \quad 
		\mathcal{E}_u(T)+\mathcal{E}_w(T) \le 3K\bar\rho^{\alpha},
	\end{equation}
	then the following estimates hold:
	\begin{equation}\label{a2}
		\mathcal{E}_\rho(T) \le 2\mathcal{E}_\rho(0), \quad 
		\mathcal{E}_u(T)+\mathcal{E}_w(T) \le 2K\bar\rho^{\alpha},
	\end{equation}
	provided 
	\begin{equation}
		\bar\rho\ge M(\Omega,C_0,\bar\mu,\bar\xi,\bar\eta,\bar\lambda,\alpha, \beta,\|\nabla\rho_0\|_{L^q}, \| u_0\|_{H^2}, \| w_0\|_{H^2}).
	\end{equation}     
\end{proposition}
First, since the density satisfies the transport equation \eqref{ins}$_1$ and using \eqref{ins}$_4$, one has the following lemma.
\begin{lemma}\label{uniform.rho}
Under the conditions of Proposition \ref{pr}, it holds that 
	\begin{equation}\label{2.3}
		\bar \rho \le \rho \le C_0 \bar \rho,\quad (x,t)\in \Omega \times [0,T].
	\end{equation}
\end{lemma}
Next, we establish the basic energy inequality of \eqref{ins}.
\begin{lemma}\label{lem:basic-est}
Under the conditions of Proposition \ref{pr} there exists a positive constant $M_1$ such that
\begin{equation}\label{basic-est}
\sup_{0\le t\le T}\bar{\rho}\left(\|u\|_{L^2}^2+\| w\|_{L^2}^2\right)+\bar\rho^{\alpha}\int_{0}^{T}\left( \|\nabla u\|_{L^2}^2+\|\nabla w\|_{L^2}^2\right)\dif t\le C\bar\rho,
\end{equation}
and
\begin{equation}\label{basic-est.2}
\sup_{0\le t\le T}\bar{\rho}e^{\kappa(\bar\rho)t}\left(\|u\|_{L^2}^2+\| w\|_{L^2}^2\right)+\bar\rho^{\alpha}\int_{0}^{T}e^{\kappa(\bar\rho)t}\left( \|\nabla u\|_{L^2}^2+\|\nabla w\|_{L^2}^2\right)\dif t\le C\bar\rho,
\end{equation}
provided $\bar \rho \geq M_1(\Omega,C_0,\bar\mu,\bar\xi,\bar\eta,\bar\lambda,\alpha,\beta,\|\nabla\rho_0\|_{L^q}, \| u_0\|_{H^2}, \| w_0\|_{H^2})$, where $\kappa(\bar{\rho})=\bar{\kappa}\bar{\rho}^{\alpha-1}$ for some constant $\bar\kappa$ independent of $\bar{\rho}$.
\end{lemma}

\begin{proof}
Making the $L^2$-inner product of \eqref{ins}$_2$--\eqref{ins}$_3$ by $(u, w)^\top$, we obtain after integration by parts that 
\begin{align}\label{b.e.1}
&\frac{1}{2}\frac{\dif}{\dif t}\left(\|\sqrt{\rho} u\|_{L^2}^2+\|\sqrt{\rho}w\|_{L^2}^2\right) + \int(\bar\mu\rho^\alpha+\bar\xi\rho^\beta)|\calD(u)|^2  \dif x\nonumber\\
&+2\bar\eta\int\rho^\alpha|\calD(w)|^2\dif x+4\bar\xi\int\rho^\beta|w|^2\dif x+\bar\lambda\int\rho^\alpha(\div w)^2\dif x\nonumber\\
=~&2\bar\xi\int\left(\rho^\beta w\cdot(\D\times u)+w\cdot\D\times (\rho^\beta u)\right)\dif x\nonumber\\
\leq~&C\bar\rho^\beta\left(\|u\|_{H^1}\|w\|_{H^1}+\|u\|_{L^{\frac{2q}{q-1}}}\|w\|_{L^{\frac{2q}{q-1}}}\|\D\rho\|_{L^q}\right)\nonumber\\
\leq~&C\bar\rho^\beta\left(\|\D u\|_{L^2}^2+\|\D w\|_{L^2}^2\right),
\end{align}
where we have used \eqref{a1}, Poincar\'e's inequality and the fact $\frac{2q}{q-1}\in(2,3)$.

Integrating the above inequality over $(0,T]$ and taking $M_1$ large enough, such that
\begin{equation*}
4C\bar\rho^\beta\left(\|\D u\|_{L^2}^2+\|\D w\|_{L^2}^2\right)\leq\bar\mu\bar\rho^\alpha\int|\calD(u)|^2\dif x+2\bar\eta\bar{\rho}^\alpha\int|\calD(w)|^2\dif x,
\end{equation*}
we obtain 
\begin{equation}
\sup_{0\le t\le T}\bar{\rho}\left(\|u\|_{L^2}^2+\| w\|_{L^2}^2\right)+\bar\rho^{\alpha}\int_{0}^{T}\left( \|\nabla u\|_{L^2}^2+\|\nabla w\|_{L^2}^2\right)\dif t\le C\bar\rho\left(\|u_0\|_{L^2}^2+\|w_0\|_{L^2}^2\right)\le C\bar\rho.
\end{equation}

Multiplying \eqref{b.e.1} by $e^{\kappa t}$ and using Poincar\'e's inequality, we obtain
\begin{align}\label{b.e.2}
&\frac{1}{2}\frac{\dif}{\dif t}\left(e^{\kappa t}\left(\|\sqrt{\rho} u\|_{L^2}^2+\|\sqrt{\rho}w\|_{L^2}^2\right)\right)+e^{\kappa t} \int(\bar\mu\rho^\alpha+\bar\xi\rho^\beta)|\calD(u)|^2\dif x\nonumber\\
&+2\bar\eta e^{\kappa t}\int\rho^\alpha|\calD(w)|^2\dif x+4\bar\xi e^{\kappa t}\int\rho^\beta|w|^2\dif x+\bar\lambda e^{\kappa t}\int\rho^\alpha(\div w)^2\dif x\nonumber\\
\leq~&C\bar\rho^\beta e^{\kappa t}\left(\|\D u\|_{L^2}^2+\|\D w\|_{L^2}^2\right)+\frac{\kappa}{2}e^{\kappa t}\left(\|\sqrt{\rho} u\|_{L^2}^2+\|\sqrt{\rho}w\|_{L^2}^2\right)\nonumber\\
\leq~&C\bar\rho^\beta e^{\kappa t}\left(\|\D u\|_{L^2}^2+\|\D w\|_{L^2}^2\right)+\frac{C_0\bar\kappa}{2}\bar{\rho}^\alpha e^{\kappa t}\left(\|\D u\|_{L^2}^2+\|\D w\|_{L^2}^2\right).
\end{align}
Integrating the above inequality over $(0,T]$ and taking $\bar\kappa$ sufficiently small, such that
\begin{equation*}
4C_0\bar\kappa\bar{\rho}^\alpha\left(\|\D u\|_{L^2}^2+\|\D w\|_{L^2}^2\right)\leq\bar\mu\bar\rho^\alpha\int|\calD(u)|^2\dif x+2\bar\eta\bar{\rho}^\alpha\int|\calD(w)|^2\dif x,
\end{equation*}
we complete the proof of Lemma \ref{lem:basic-est}.
\end{proof}

Next, we have the following high-order estimate of the velocities which will be used frequently.

\begin{lemma}\label{lem:D2u}
Under the conditions of Proposition \ref{pr}, it holds that 
	\begin{equation}\label{H2}
	    \| u\|_{H^2} \leq C\left(\bar\rho^{\frac{1}{2}-\alpha}\|\sqrt{\rho} u_t\|_{L^2}
+\bar\rho^{2-2\alpha}\|\nabla u\|_{L^2}^3+\bar{\rho}^{\beta-\alpha}\|\D\times w\|_{L^2}\right),
	\end{equation}
    \begin{equation}\label{H2w}
	    \| w\|_{H^2} \leq C\left(\bar\rho^{\frac{1}{2}-\alpha}\|\sqrt{\rho} w_t\|_{L^2}
+\bar\rho^{2-2\alpha}\|\nabla u\|_{L^2}^2\|\nabla w\|_{L^2}+\bar{\rho}^{\beta-\alpha}\|\D u\|_{L^2}\right),
	\end{equation}
        \begin{equation}\label{W2q}
	    \| u\|_{W^{2,q}} \leq C \Big(\bar\rho^{-\alpha}  \|\rho u_t\|_{L^q} +  \bar \rho^{(1-\alpha)\frac{5q-6}{q}} \|\nabla u\|_{L^2}^{\frac{6(q-1)}{q}}+\bar{\rho}^{\beta-\alpha}\|\D\times w\|_{L^q}\Big),
	\end{equation}	
    and
    \begin{equation}\label{W2qw}
	    \| w\|_{W^{2,q}} \leq C \Big(\bar\rho^{-\alpha}  \|\rho w_t\|_{L^q} +  \bar \rho^{(1-\alpha)\frac{5q-6}{q}} \|\nabla u\|_{L^2}^{\frac{5q-6}{q}}\|\D w\|_{L^2}+\bar{\rho}^{\beta-\alpha}\|\D u\|_{L^q}\Big).
	\end{equation}
\end{lemma}
\begin{proof}
Let $F=-\rho u_t-\rho(u\cdot\D)u+2\xi(\rho)\D\times w$ in Lemma \ref{stokes-e}, then the proof of Lemma \ref{lem:D2u} is similar to that of \cite{Huang-Li-Zhang-arxiv}, and we omit it here.
\end{proof}

Now we are ready to derive the estimates of $\mathcal{E}_u(T)$ and $\mathcal{E}_w(T)$.
\begin{lemma}\label{L_2}
Under the conditions of Proposition \ref{pr}, there exists a positive constant $M_2$ such that
\begin{equation}\label{Euw}
\mathcal{E}_u(T)+\mathcal{E}_w(T) \le 2K\bar\rho^{\alpha},
\end{equation}
\begin{equation}\label{tdu}
\sup_{t\in[0,T]}\bar{\rho}^\alpha t\left(\|\nabla u \|_{L^2}^2+\|\D w\|_{L^2}^2\right)
+\int_0^{T}t\left(\|\sqrt{\rho} u_t \|_{L^2}^2+\|\sqrt{\rho} w_t \|_{L^2}^2\right)\dif t
\leq C\bar\rho ,
\end{equation}
\begin{equation}\label{t2du}
\sup_{t\in[0,T]}\bar{\rho}^\alpha t^2\left(\|\nabla u \|_{L^2}^2+\|\D w\|_{L^2}^2\right)
+\int_0^{T}t^2\left(\|\sqrt{\rho} u_t \|_{L^2}^2+\|\sqrt{\rho} w_t \|_{L^2}^2\right)\dif t
\leq C\bar\rho ,
\end{equation}
and
\begin{equation}\label{etdu}
\sup_{t\in[0,T]}\bar{\rho}^\alpha e^{\kappa(\bar{\rho})t}\left(\|\nabla u \|_{L^2}^2+\|\D w\|_{L^2}^2\right)
+\int_0^{T}e^{\kappa(\bar{\rho})t}\left(\|\sqrt{\rho} u_t \|_{L^2}^2+\|\sqrt{\rho} w_t \|_{L^2}^2\right)\dif t
\leq C\bar{\rho}^\alpha,
\end{equation}
provided $\bar \rho \geq M_2(\Omega,C_0,\bar\mu,\bar\xi,\bar\eta,\bar\lambda,\alpha,\beta,\|\nabla\rho_0\|_{L^q}, \| u_0\|_{H^2}, \| w_0\|_{H^2})$.
\end{lemma}
\begin{proof}
Making the $L^2$-inner product of \eqref{ins}$_2$--\eqref{ins}$_3$ with $(u_t, w_t)^\top$ and integrating by parts, we have 
\begin{align}\label{51}
&\frac{1}{2}\frac{\dif}{\dif t}\left(\int \left(\bar\mu\rho^\alpha+\bar\xi\rho^\beta\right) |\calD(u)|^2\dif x+2\bar\eta\int\rho^\alpha|\calD(w)|^2\dif x+\bar\lambda\int\rho^\alpha(\div w)^2\dif x\right)\nonumber\\
&+2\bar\xi\frac{\dif}{\dif t}\int\rho^\beta|w|^2\dif x+  \int\rho (|u_t|^2+|w_t|^2)\dif x\nonumber \\
=~& -\int  \rho u \cdot \nabla u \cdot u_t \dif x -\int  \rho u \cdot \nabla w \cdot w_t \dif x+2\bar\xi\int\rho^\beta\left(u_t\cdot\D\times w+w_t\cdot\D\times u\right)\dif x\nonumber\\
&+\frac{1}{2}\int(\alpha\bar\mu\rho^{\alpha-1}+\beta\bar\xi\rho^{\beta-1})\rho_t|\calD(u)|^2\dif x+\frac{\bar\eta\alpha}{2}\int\rho^{\alpha-1}\rho_t|\calD(w)|^2\dif x\nonumber\\
&+\frac{\bar\lambda\alpha}{2}\int\rho^{\alpha-1}\rho_t(\div w)^2\dif x+2\bar\xi\beta\int\rho^{\beta-1}\rho_t|w|^2\dif x \nonumber\\
:=~&\sum\limits_{j=1}^7I_j.
\end{align}
It follows from H\"older and Sobolev inequalities that 
\begin{align}\label{52}
I_1+I_2=~&-\int  \rho u \cdot \nabla u \cdot u_t \dif x-\int  \rho u \cdot \nabla w \cdot w_t \dif x\nonumber\\
\leq~& C\bar\rho^\frac{1}{2}\|u\|_{L^6}\left(\|\sqrt{\rho} u_t\|_{L^2}\|\nabla u\|_{L^3}+\|\sqrt{\rho} w_t\|_{L^2}\|\nabla w\|_{L^3}\right)\nonumber \\
\leq~& C\bar\rho^\frac{1}{2}\|\sqrt{\rho}u_t\|_{L^2}\|\nabla u\|_{L^2}^\frac{3}{2}\|\nabla u\|_{H^1}^\frac{1}{2}+C\bar\rho^\frac{1}{2}\|\sqrt{\rho} w_t\|_{L^2}\|\nabla u\|_{L^2}\|\nabla w\|_{L^2}^\frac{1}{2}\|\nabla w\|_{H^1}^\frac{1}{2}\nonumber\\
\leq~& C\bar\rho^\frac{1}{2}\|\sqrt{\rho} u_t\|_{L^2}\|\nabla u\|_{L^2}^\frac{3}{2}\left(\bar\rho^{\frac{1}{2}-\alpha}\|\sqrt{\rho} u_t\|_{L^2}+\bar\rho^{2-2\alpha}\|\nabla u\|_{L^2}^3+\bar{\rho}^{\beta-\alpha}\|\D w\|_{L^2}\right)^\frac{1}{2}\nonumber\\
&+C\bar\rho^\frac{1}{2}\|\sqrt{\rho} w_t\|_{L^2}\|\nabla u\|_{L^2}\|\nabla w\|_{L^2}^\frac{1}{2}\left(\bar\rho^{\frac{1}{2}-\alpha}\|\sqrt{\rho}w_t\|_{L^2}
+\bar\rho^{2-2\alpha}\|\nabla u\|_{L^2}^2\|\D w\|_{L^2}+\bar{\rho}^{\beta-\alpha}\|\D u\|_{L^2}\right)^\frac{1}{2}\nonumber\\
\leq~& \frac{1}{6}\left(\|\sqrt{\rho} u_t\|_{L^2}^2+\|\sqrt{\rho} w_t\|_{L^2}^2\right)
+C\bar\rho^{3-2\alpha}\left(\|\nabla u\|_{L^2}^6+\|\D u\|_{L^2}^4\|\D w\|_{L^2}^2\right)+C\bar{\rho}^{1+\beta-\alpha}\|\D u\|_{L^2}^3\|\D w\|_{L^2}\nonumber\\
\leq~& \frac{1}{6}\left(\|\sqrt{\rho} u_t\|_{L^2}^2+\|\sqrt{\rho} w_t\|_{L^2}^2\right)+C\left(\bar{\rho}^{3-2\alpha}+\bar{\rho}^{1+\beta-\alpha}\right)\|\nabla u\|_{L^2}^2,
\end{align}
where we have also used \eqref{H2}, \eqref{H2w} and $\bar{\rho}>1$. Next, we have
\begin{align}\label{53}
I_3=~&2\bar\xi\int\rho^\beta\left(u_t\cdot\D\times w+w_t\cdot\D\times u\right)\dif x \nonumber\\
\leq~& C \bar{\rho}^{\beta-\frac{1}{2}}\left(\|\sqrt{\rho}u_t\|_{L^2}\|\D w\|_{L^2}+\|\sqrt{\rho}w_t\|_{L^2}\|\D u\|_{L^2}\right)\nonumber\\
\leq~& \frac{1}{6}\left(\|\sqrt{\rho} u_t\|_{L^2}^2+\|\sqrt{\rho} w_t\|_{L^2}^2\right)+C\bar{\rho}^{2\beta-1}\left(\|\D u\|_{L^2}^2+\|\D w\|_{L^2}^2\right).
\end{align}
Using \eqref{ins}$_1$, \eqref{H2}, H\"older and Sobolev inequalities, we get
\begin{align}\label{54}
I_4=~&-\frac{1}{2}\int(\alpha\bar\mu\rho^{\alpha-1}+\beta\bar\xi\rho^{\beta-1})u\cdot\D\rho|\calD(u)|^2\dif x\nonumber\\
\leq~&C\bar{\rho}^{\alpha-1}\|u\|_{L^{\frac{2q}{q-2}}}\|\D\rho\|_{L^q}\|\D u\|_{L^4}^2\nonumber\\
\leq~&C\bar{\rho}^{\alpha-1}\|\D u\|_{L^2}^{\frac{1}{2}}\|\D^2u\|_{L^2}^{\frac{3}{2}}\nonumber\\
\leq~&C\bar{\rho}^{\alpha-1}\|\D u\|_{L^2}^{\frac{1}{2}}\left(\bar\rho^{\frac{1}{2}-\alpha}\|\sqrt{\rho}u_t\|_{L^2}+\bar\rho^{2-2\alpha}\|\D u\|_{L^2}^3+\bar{\rho}^{\beta-\alpha}\|\D w\|_{L^2}\right)^{\frac{3}{2}}\nonumber\\
\leq~&\frac{1}{6}\|\sqrt{\rho} u_t\|_{L^2}^2+C\bar{\rho}^{-1-2\alpha}\|\D u\|_{L^2}^2+\bar{\rho}^{2-2\alpha}\|\D u\|_{L^2}^5+C\bar{\rho}^{\frac{3}{2}\beta-\frac{\alpha}{2}-1}\|\D u\|_{L^2}^{\frac{1}{2}}\|\D w\|_{L^2}^{\frac{3}{2}}\nonumber\\
\leq~&\frac{1}{6}\|\sqrt{\rho} u_t\|_{L^2}^2+C\bar{\rho}^{2-2\alpha}\|\D u\|_{L^2}^2+C\bar{\rho}^{\frac{3}{2}\beta-\frac{\alpha}{2}-1}\|\D u\|_{L^2}^{\frac{1}{2}}\|\D w\|_{L^2}^{\frac{3}{2}}.
\end{align}
Similarly, it follows from \eqref{ins}$_1$ and \eqref{H2w} that
\begin{align}\label{55}
I_5+I_6=~&\frac{\bar\eta\alpha}{2}\int\rho^{\alpha-1}\rho_t|\calD(w)|^2\dif x+\frac{\bar\lambda\alpha}{2}\int\rho^{\alpha-1}\rho_t(\div w)^2\dif x\nonumber\\
\leq~&C\bar{\rho}^{\alpha-1}\|u\|_{L^{\frac{2q}{q-2}}}\|\D\rho\|_{L^q}\|\D w\|_{L^4}^2\nonumber\\
\leq~&C\bar{\rho}^{\alpha-1}\|\D w\|_{L^2}^{\frac{1}{2}}\|\D^2w\|_{L^2}^{\frac{3}{2}}\nonumber\\
\leq~&C\bar{\rho}^{\alpha-1}\|\D w\|_{L^2}^{\frac{1}{2}}\left(\bar\rho^{\frac{1}{2}-\alpha}\|\sqrt{\rho}w_t\|_{L^2}+\bar\rho^{2-2\alpha}\|\D u\|_{L^2}^2\|\D w\|_{L^2}+\bar{\rho}^{\beta-\alpha}\|\D u\|_{L^2}\right)^{\frac{3}{2}}\nonumber\\
\leq~&\frac{1}{6}\|\sqrt{\rho} w_t\|_{L^2}^2+C\bar{\rho}^{-1-2\alpha}\|\D w\|_{L^2}^2+\bar{\rho}^{2-2\alpha}\|\D u\|_{L^2}^3\|\D w\|_{L^2}^2+C\bar{\rho}^{\frac{3}{2}\beta-\frac{\alpha}{2}-1}\|\D u\|_{L^2}^{\frac{3}{2}}\|\D w\|_{L^2}^{\frac{1}{2}}\nonumber\\
\leq~&\frac{1}{6}\|\sqrt{\rho} w_t\|_{L^2}^2+C\bar{\rho}^{2-2\alpha}\|\D w\|_{L^2}^2+C\bar{\rho}^{\frac{3}{2}\beta-\frac{\alpha}{2}-1}\|\D u\|_{L^2}^{\frac{3}{2}}\|\D w\|_{L^2}^{\frac{1}{2}}.
\end{align}
Next, it follows from \eqref{ins}$_1$ and Poincar\'e's inequality that
\begin{align}\label{56}
I_7\leq C\bar{\rho}^{\alpha-1}\|u\|_{L^{\frac{2q}{q-2}}}\|\D\rho\|_{L^q}\|w\|_{L^4}^2\leq C\bar{\rho}^{\alpha-1}\|w\|_{L^2}^{\frac{1}{2}}\|\D w\|_{L^2}^{\frac{3}{2}}\leq C\bar{\rho}^{\alpha-1}\|\D w\|_{L^2}^2.
\end{align}

Substituting \eqref{52}--\eqref{55} into \eqref{51}, we obtain
\begin{align}\label{t1}
&\frac{1}{2}\frac{\dif}{\dif t}\left(\int \left(\bar\mu\rho^\alpha+\bar\xi\rho^\beta\right) |\calD(u)|^2\dif x+2\bar\eta\int\rho^\alpha|\calD(w)|^2\dif x+\bar\lambda\int\rho^\alpha(\div w)^2\dif x\right)\nonumber\\
&+2\bar\xi\frac{\dif}{\dif t}\int\rho^\beta|w|^2\dif x+\frac{1}{2}\int\rho (|u_t|^2+|w_t|^2)\dif x\nonumber \\
\leq~& C\left(\bar{\rho}^{3-2\alpha}+\bar{\rho}^{1+\beta-\alpha}\right)\|\nabla u\|_{L^2}^2+C\bar{\rho}^{2\beta-1}\left(\|\D u\|_{L^2}^2+\|\D w\|_{L^2}^2\right)+C\bar{\rho}^{2-2\alpha}\|\D u\|_{L^2}^2\nonumber\\
&+C\bar{\rho}^{\frac{3}{2}\beta-\frac{\alpha}{2}-1}\|\D u\|_{L^2}^{\frac{1}{2}}\|\D w\|_{L^2}^{\frac{3}{2}}+C\bar{\rho}^{2-2\alpha}\|\D w\|_{L^2}^2+C\bar{\rho}^{\frac{3}{2}\beta-\frac{\alpha}{2}-1}\|\D u\|_{L^2}^{\frac{3}{2}}\|\D w\|_{L^2}^{\frac{1}{2}}+C\bar{\rho}^{\alpha-1}\|\D w\|_{L^2}^2\nonumber\\
\leq~&C\left(\bar{\rho}^{3-2\alpha}+\bar{\rho}^{1+\beta-\alpha}+\bar{\rho}^{2\beta-1}+\bar{\rho}^{\alpha-1}\right)\left(\|\D u\|_{L^2}^2+\|\D w\|_{L^2}^2\right).
\end{align}

Integrating \eqref{t1} with respect to $t$ over $(0, T]$, we obtain from \eqref{2.3} and \eqref{basic-est} that
\begin{align*}
&\sup_{t\in[0,T]}\bar\rho^{\alpha}\left(\|\D u\|_{L^2}^2+\|\D w\|_{L^2}^2\right)
+\int_0^T\left(\|\sqrt{\rho}u_t\|_{L^2}^2+\|\sqrt{\rho}w_t\|_{L^2}^2\right)\dif t \\
\leq~& C_1\bar\rho^{\alpha}
 +C\left(\bar{\rho}^{3-2\alpha}+\bar{\rho}^{1+\beta-\alpha}+\bar{\rho}^{2\beta-1}+\bar{\rho}^{\alpha-1}\right)\int_0^T\left(\|\D u\|_{L^2}^2+\|\D w\|_{L^2}^2\right)\dif t\\
\leq~& C_1\bar\rho^{\alpha}+C_2\left(\bar{\rho}^{4-3\alpha}+\bar{\rho}^{2+\beta-2\alpha}+\bar{\rho}^{2\beta-\alpha}+1\right).
\end{align*}
Taking $K\geq C_1$ and $M_2$ sufficiently large, such that
\begin{equation}
C_2\left(\bar{\rho}^{4-3\alpha}+\bar{\rho}^{2+\beta-2\alpha}+\bar{\rho}^{2\beta-\alpha}+1\right)\leq C_1\bar\rho^{\alpha},
\end{equation}
we complete the proof of \eqref{Euw}.

For simplicity, we denote
\begin{equation}\label{calE}
\calE(t):=\int\left(\bar\mu\rho^\alpha+\bar\xi\rho^\beta\right) |\calD(u)|^2\dif x+2\bar\eta\int\rho^\alpha|\calD(w)|^2\dif x+\bar\lambda\int\rho^\alpha(\div w)^2\dif x+4\bar\xi\int\rho^\beta|w|^2\dif x.
\end{equation}
It's easy to see that there exist a positive constant $C$ independent of $\bar{\rho}$, such that
\begin{equation}\label{calE:equi}
C^{-1}\bar\rho^{\alpha}\left(\|\nabla u \|_{L^2}^2+\|\nabla w \|_{L^2}^2\right)\leq\calE(t)\leq C\bar\rho^{\alpha}\left(\|\nabla u \|_{L^2}^2+\|\nabla w \|_{L^2}^2\right),
\end{equation}
and \eqref{t1} can be rewritten as
\begin{equation}\label{t1.1}
\frac{\dif}{\dif t}\calE(t)+\left(\|\sqrt{\rho} u_t\|_{L^2}^2+\|\sqrt{\rho} w_t\|_{L^2}^2\right)\leq C\left(\bar{\rho}^{3-2\alpha}+\bar{\rho}^{1+\beta-\alpha}+\bar{\rho}^{2\beta-1}+\bar{\rho}^{\alpha-1}\right)\left(\|\D u\|_{L^2}^2+\|\D w\|_{L^2}^2\right).
\end{equation}
Multiplying \eqref{t1.1} by $t^k$ for $k=1,2$, we obtain
\begin{align}\label{t1.2}
&\frac{\dif}{\dif t}(t^k\calE(t))+t^k\left(\|\sqrt{\rho} u_t\|_{L^2}^2+\|\sqrt{\rho} w_t\|_{L^2}^2\right)\nonumber\\
\leq~& C(t^k+t^{k-1})\left(\bar{\rho}^{3-2\alpha}+\bar{\rho}^{1+\beta-\alpha}+\bar{\rho}^{2\beta-1}+\bar{\rho}^{\alpha-1}\right)\left(\|\D u\|_{L^2}^2+\|\D w\|_{L^2}^2\right)\nonumber\\
\leq~&C(t^k+t^{k-1})\bar{\rho}^\alpha\left(\|\D u\|_{L^2}^2+\|\D w\|_{L^2}^2\right).
\end{align}
Integrating \eqref{t1.2} with respect to $t$ over $(0, T]$ and using \eqref{basic-est.2}, we obtain \eqref{tdu} and \eqref{t2du}.

On the other hand, multiplying \eqref{t1.1} by $e^{\kappa(\bar{\rho})t}$, we get
\begin{align}\label{t1.3}
&\frac{\dif}{\dif t}\left(e^{\kappa(\bar{\rho})t}\calE(t)\right)+e^{\kappa(\bar{\rho})t}\left(\|\sqrt{\rho} u_t\|_{L^2}^2+\|\sqrt{\rho} w_t\|_{L^2}^2\right)\nonumber\\
\leq~& Ce^{\kappa(\bar{\rho})t}(1+\bar{\kappa}\bar{\rho}^{\alpha-1})\left(\bar{\rho}^{3-2\alpha}+\bar{\rho}^{1+\beta-\alpha}+\bar{\rho}^{2\beta-1}+\bar{\rho}^{\alpha-1}\right)\nonumber\\
\leq~&Ce^{\kappa(\bar{\rho})t}\left(\bar{\rho}^{2-\alpha}+\bar{\rho}^\beta+\bar{\rho}^{\alpha+2\beta-2}+\bar{\rho}^{2\alpha-2}\right)\left(\|\D u\|_{L^2}^2+\|\D w\|_{L^2}^2\right).
\end{align}
Integrating \eqref{t1.3} with respect to $t$ over $(0, T]$ and using \eqref{basic-est.2}, we obtain \eqref{etdu}. This completes the proof of Lemma \ref{L_2}.
\end{proof}

In order to close the estimate of $\mathcal{E}_\rho$, we need the following time-weight estimates.
\begin{lemma}
Under the conditions of Proposition \ref{pr}, there exists a positive constant $M_3$ such that
\begin{equation}\label{trut}
\sup_{0\le t\le T}t\left(\|\sqrt{\rho}u_t\|_{L^2}^2+\|\sqrt{\rho}w_t\|_{L^2}^2\right) +  \bar\rho^\alpha\int_{0}^{T} t\left(\|\nabla u_t\|^2_{L^2}+\|\nabla w_t\|^2_{L^2}\right) \dif t \leq C\bar\rho^\alpha,
\end{equation}
and
\begin{equation}\label{ttrut}
\sup_{0\le t\le T}t^2\left(\|\sqrt{\rho}u_t\|_{L^2}^2+\|\sqrt{\rho}w_t\|_{L^2}^2\right) +  \bar\rho^\alpha\int_{0}^{T} t^2\left(\|\nabla u_t\|^2_{L^2}+\|\nabla w_t\|^2_{L^2}\right) \dif t \leq C\bar\rho.
\end{equation}
provided $\bar \rho \geq M_3(\Omega,C_0,\bar\mu,\bar\xi,\bar\eta,\bar\lambda,\alpha,\beta,\|\nabla\rho_0\|_{L^q}, \| u_0\|_{H^2}, \| w_0\|_{H^2})$.
\end{lemma}
\begin{proof}
Taking the $t$-derivative of \eqref{ins}$_2$--\eqref{ins}$_3$ and multiply the resulting equation by $(tu_t, tw_t)^\top$, after integration by parts, we obtain
\begin{align}\label{k2}
&\frac{t}{2}\frac{\dif}{\dif t}\int\rho(|u_t|^2+|w_t|^2)\dif x+2t\int\left( \bar{\mu}\rho^\alpha+\bar{\xi}\rho^\beta\right)|\calD(u_t)|^2\dif x \nonumber \\
&+4\bar{\xi}t\int\rho^\beta|w_t|^2\dif x+2\bar{\eta}t\int\rho^\alpha|\calD(w_t)|^2\dif x+\bar{\lambda}t\int\rho^\alpha(\div w_t)^2\dif x \nonumber\\
=~&t\int u\cdot\D\rho(|u_t|^2+|w_t|^2)\dif x-t\int\rho u_t\cdot\nabla u\cdot u_t \dif x-t\int\rho u_t\cdot\nabla w\cdot w_t\dif x \nonumber \\
&+t\int(u\cdot\D\rho)(u\cdot\nabla u\cdot u_t)\dif x+t\int(u\cdot\D\rho)(u\cdot\nabla w\cdot w_t)\dif x+4\beta\bar{\xi}t\int\rho^{\beta-1}(u\cdot\D\rho)w\cdot w_t\dif x \nonumber \\
&+2t\int\left(\alpha\bar{\mu}\rho^{\alpha-1}+\beta\bar{\xi}\rho^{\beta-1}\right)(u\cdot\D\rho)\calD(u):\D u_t\dif x \nonumber \\
&+2\alpha\bar{\eta}t\int\rho^{\alpha-1}(u\cdot\D\rho)\calD(w):\D w_t\dif x+\alpha\bar{\lambda}t\int\rho^{\alpha-1}(u\cdot\D\rho)\div w\div w_t\dif x \nonumber \\
&-2\beta\bar{\xi}t\int\rho^{\beta-1}(u\cdot\D\rho)\left(u_t\cdot(\D\times w)+w_t\cdot(\D\times u)\right)\dif x+2\bar{\xi}t\int\rho^\beta\left(u_t\cdot(\D\times w_t)+w_t\cdot(\D\times u_t)\right)\dif x \nonumber\\
=:~&\sum_{i=1}^{11}J_i,
\end{align}
where the terms $J_i$ ($i=1, \cdots, 11$) can be estimated as follows.
\begin{align}\label{J1}
J_1\leq~& Ct\int|\nabla\rho\cdot u|(|u_t|^2+|w_t|^2)\dif x \nonumber\\
\leq~&C t\|\nabla \rho\|_{L^q} \|u \|_{L^{\frac{2q}{q-2}}} (\| u_t \|_{L^4}^{2}+\|w_t\|_{L^4}^2) \nonumber\\
\leq~&Ct \bar \rho^{-\frac14} \|\nabla \rho\|_{L^q} \|u \|_{L^{2}}^{\frac{q-3}{q}} \|\nabla u \|_{L^{2}}^{\frac{3}{q}} \left(\|\sqrt{\rho} u_t \|_{L^2}^{\frac12} \|\nabla u_t \|_{L^2}^{\frac32}+\|\sqrt{\rho}w_t\|_{L^2}^{\frac{1}{2}}\|\D w_t \|_{L^2}^{\frac{3}{2}}\right) \nonumber\\
\le~& \frac{\bar{\mu}}{10}t\bar{\rho}^\alpha\|\D u_t\|_{L^2}^2+\frac{\bar{\eta}}{10}t\bar{\rho}^\alpha\|\D w_t\|_{L^2}^2+Ct\bar{\rho}^{-3\alpha-1}\|\D\rho\|_{L^q}^4\|\D u\|_{L^{2}}^{\frac{12}{q}}(\|\sqrt{\rho}u_t \|_{L^2}^2+\|\sqrt{\rho}w_t\|_{L^2}^2) \nonumber\\
\leq~& \frac{\bar\mu}{10} t \bar\rho^{\alpha} \|\nabla u_t\|_{L^2}^2+\frac{\bar\eta}{10} t \bar\rho^{\alpha} \|\nabla w_t\|_{L^2}^2+Ct\bar{\rho}^{-3\alpha-1}\|\D u\|_{L^{2}}^{\frac{12}{q}}(\|\sqrt{\rho}u_t \|_{L^2}^2+\|\sqrt{\rho}w_t\|_{L^2}^2),
\end{align}
\begin{align}\label{J23}
J_2+J_3\leq~& Ct\int(|\rho\nabla u||u_t|^2+|\rho\D w||u_t||w_t|)\dif x \nonumber \\
\leq~&Ct\bar\rho^{\frac{1}{2}}\|\sqrt{\rho}u_t\|_{L^2}(\|\D u\|_{L^3}\|u_t\|_{L^6}+\|\D w\|_{L^3}\|w_t\|_{L^6}) \nonumber \\
\leq~&Ct\bar\rho^{\frac{1}{2}}\|\sqrt{\rho}u_t\|_{L^2}(\|\D u\|_{L^2}^{\frac{1}{2}}\|\D u\|_{L^6}^{\frac{1}{2}}\|\D u_t\|_{L^2}+\|\D w\|_{L^2}^{\frac{1}{2}}\|\D w\|_{L^6}^{\frac{1}{2}}\|\D w_t\|_{L^2}) \nonumber \\
\leq~& \frac{\bar{\mu}}{10}t\bar{\rho}^\alpha\|\D u_t\|_{L^2}^2+\frac{\bar{\eta}}{10}t\bar{\rho}^\alpha\|\D w_t\|_{L^2}^2 +Ct\bar{\rho}^{1-\alpha}\|\sqrt{\rho}u_t\|_{L^2}^2(\|\nabla u\|_{L^2} \|\nabla u\|_{H^1}+\|\nabla w\|_{L^2}\|\nabla w\|_{H^1}) \nonumber \\
\leq~& \frac{\bar{\mu}}{10}t\bar{\rho}^\alpha\|\D u_t\|_{L^2}^2+\frac{\bar{\eta}}{10}t\bar{\rho}^\alpha\|\D w_t\|_{L^2}^2+Ct\bar\rho^{\frac{3}{2}-2\alpha}\|\nabla u\|_{L^2} \|\sqrt{\rho}u_t\|_{L^2}^3+Ct\bar\rho^{3-3\alpha}\|\nabla u\|_{L^2}^4 \|\sqrt{\rho}u_t\|_{L^2}^2 \nonumber\\
&+Ct\bar\rho^{1+\beta-2\alpha} \|\nabla u\|_{L^2}\|\nabla w\|_{L^2} \|\sqrt{\rho} u_t \|_{L^2}^2+Ct\bar\rho^{\frac{3}{2}-2\alpha}\|\nabla w\|_{L^2}\|\sqrt{\rho}u_t\|_{L^2}^2\|\sqrt{\rho} w_t\|_{L^2} \nonumber \\
&+Ct\bar\rho^{3-3\alpha} \|\nabla u\|_{L^2}^2\|\nabla w\|_{L^2}^2 \|\sqrt{\rho} u_t \|_{L^2}^2,
\end{align}
where we have used \eqref{H2} and \eqref{H2w}. Similarly, we have
\begin{align}\label{J45}
J_4+J_5\leq~&Ct\int|\D\rho||u|^2(|\D u||u_t|+|\D w||w_t|)\dif x \nonumber \\
\leq~& Ct\|\D\rho\|_{L^q}\|u\|_{L^6}^2\left(\|\D u\|_{L^{\frac{2q}{q-2}}}\|u_t\|_{L^6}+\|\D w\|_{L^{\frac{2q}{q-2}}}\|w_t\|_{L^6}\right) \nonumber \\
\leq~& \frac{\bar{\mu}}{10}t\bar{\rho}^\alpha\|\D u_t\|_{L^2}^2+\frac{\bar{\eta}}{10}t\bar{\rho}^\alpha\|\D w_t\|_{L^2}^2+ Ct \bar\rho^{ - \alpha} \|\nabla \rho \|_{L^q}^2 \|\nabla u\|_{L^2}^\frac{6(q-1)}{q}\|\nabla u\|_{H^1}^\frac{6}{q} \nonumber \\
&+Ct\bar{\rho}^{-\alpha}\|\D\rho\|_{L^q}^2\|\D u\|_{L^2}^4\|\D w\|_{L^2}^{2-\frac{6}{q}}\|\D w\|_{H^1}^\frac{6}{q} \nonumber\\
\leq~& \frac{\bar{\mu}}{10}t\bar{\rho}^\alpha\|\D u_t\|_{L^2}^2+\frac{\bar{\eta}}{10}t\bar{\rho}^\alpha\|\D w_t\|_{L^2}^2+Ct\bar\rho^{ -\alpha+\frac{6}{q}(\frac{1}{2}-\alpha)}\|\D u\|_{L^2}^{\frac{6(q-1)}{q}}\|\sqrt{\rho}u_t\|_{L^2}^{\frac{6}{q}}\nonumber\\
&+Ct\rho^{-\alpha+2(1-\alpha)\frac{6}{q}}\|\nabla u\|_{L^2}^{6+\frac{12}{q}}+Ct\bar\rho^{-\alpha+\frac{6}{q}(\beta-\alpha)}\|\D u\|_{L^2}^\frac{6(q-1)}{q}\|\D w\|_{L^2}^{\frac{6}{q}} \nonumber\\
&+Ct\bar\rho^{-\alpha+\frac{6}{q}(\frac{1}{2}-\alpha)}\|\D u\|_{L^2}^4\|\D w\|_{L^2}^{2-\frac{6}{q}}\|\sqrt{\rho}w_t\|_{L^2}^{\frac{6}{q}}\nonumber\\
&+Ct\rho^{-\alpha+2(1-\alpha)\frac{6}{q}}\|\nabla u\|_{L^2}^{4+\frac{12}{q}}\|\nabla w\|_{L^2}^2+Ct\bar\rho^{-\alpha+\frac{6}{q}(\beta-\alpha)}\|\D u\|_{L^2}^{4+\frac{6}{q}}\|\D w\|_{L^2}^{2-\frac{6}{q}} \nonumber\\
\leq~& \frac{\bar{\mu}}{10}t\bar{\rho}^\alpha\|\D u_t\|_{L^2}^2+\frac{\bar{\eta}}{10}t\bar{\rho}^\alpha\|\D w_t\|_{L^2}^2+Ct\left(\bar{\rho}^{ -\alpha+2(1-\alpha)\frac{6}{q}}+\bar{\rho}^{-\alpha+\frac{6}{q}(\beta-\alpha)}\right)\|\nabla u\|_{L^2}^4 \\
&+Ct\bar\rho^{ -\alpha+\frac{6}{q}(\frac{1}{2}-\alpha)}\left(\|\nabla u\|_{L^2}^2(\|\sqrt{\rho} u_t\|_{L^2}^2+\|\sqrt{\rho} w_t\|_{L^2}^2)+\|\nabla u\|_{L^2}^{4+\frac{6}{q-3}}(\|\nabla u\|_{L^2}^2+\|\nabla w\|_{L^2}^2)\right), \nonumber
\end{align}
\begin{align}\label{J678910}
&J_6+J_7+J_8+J_9+J_{10} \nonumber\\
\leq~&Ct\bar{\rho}^{\alpha-1}\int|u||\D\rho|\left(|w||w_t|+|\D w||\D w_t|+|\D u||\D u_t|+|\D w||u_t|+|\D u||w_t|\right)\dif x\nonumber\\
\leq~&Ct\bar{\rho}^{\alpha-1}\|\D\rho\|_{L^q}\|u\|_{L^\frac{3q}{q-3}}\left(\|\D u\|_{L^6}\|\D u_t\|_{L^2}+\|\D w\|_{L^6}\|\D w_t\|_{L^2}+\|\D u\|_{L^2}\|w_t\|_{L^6}+\|\D w\|_{L^2}\|u_t\|_{L^6}\right) \nonumber\\
\leq~&\frac{\bar{\mu}}{10}t\bar{\rho}^\alpha\|\D u_t\|_{L^2}^2+\frac{\bar{\eta}}{10}t\bar{\rho}^\alpha\|\D w_t\|_{L^2}^2+Ct\bar \rho^{\alpha-2} \|\nabla \rho \|_{L^q}^2  \|\nabla u\|_{L^2}^{3-\frac{6}{q}} \|\nabla u\|_{H^1}^{\frac{6}{q}-1}(\|\D u\|_{H^1}^2+\|\D w\|_{H^1}^2) \nonumber\\
\leq~&\frac{\bar{\mu}}{10}t\bar{\rho}^\alpha\|\D u_t\|_{L^2}^2+\frac{\bar{\eta}}{10}t\bar{\rho}^\alpha\|\D w_t\|_{L^2}^2+Ct\bar{\rho}^{-\frac{3}{2}+(\frac{1}{2}-\alpha)\frac{6}{q}}\|\nabla u\|_{L^2}^{3-\frac{6}{q}}\|\sqrt{\rho} u_t\|_{L^2}^{1+\frac{6}{q}}+Ct \bar\rho^{-\alpha+2(1-\alpha)\frac{6}{q}}\|\nabla u\|_{L^2}^4 \nonumber\\
&+Ct\bar{\rho}^{\beta-2+(\beta-\alpha)\frac{6}{q}}\|\nabla u\|_{L^2}^{3-\frac{6}{q}}\|\D w\|_{L^2}^{1+\frac{6}{q}}+Ct\bar{\rho}^{-\frac{3}{2}+(\frac{1}{2}-\alpha)\frac{6}{q}}\|\D u\|_{L^2}^{3-\frac{6}{q}}\|\sqrt{\rho}u_t\|_{L^2}^{\frac{6}{q}-1}\|\sqrt{\rho}w_t\|_{L^2}^2 \nonumber\\
&+Ct\bar{\rho}^{\beta-2+(\beta-\alpha)\frac{6}{q}}\|\nabla u\|_{L^2}^{5-\frac{6}{q}}\|\D w\|_{L^2}^{\frac{6}{q}-1}.
\end{align}
Next, taking $M_3$ large enough, such that
\begin{equation*}
C\bar{\rho}^\beta\leq\frac{1}{10}\min\{\bar{\mu}, \bar{\eta}\}\bar{\rho}^\alpha,
\end{equation*}
we have
\begin{align}\label{J11}
J_{11}\leq~&Ct\bar{\rho}^\beta(\|u_t\|_{L^2}\|\D w_t\|_{L^2}+\|w_t\|_{L^2}\|\D u_t\|_{L^2})\nonumber\\
\leq~&Ct\bar{\rho}^\beta\|\D u_t\|_{L^2}\|\D w_t\|_{L^2} \nonumber\\
\leq~&Ct\bar{\rho}^\beta(\|\D u_t\|_{L^2}^2+\|\D w_t\|_{L^2}^2) \nonumber\\
\leq~&\frac{\bar{\mu}}{10}t\bar{\rho}^\alpha\|\D u_t\|_{L^2}^2+\frac{\bar{\eta}}{10}t\bar{\rho}^\alpha\|\D w_t\|_{L^2}^2.
\end{align}

Substituting \eqref{J1}--\eqref{J11} into \eqref{k2}, then using \eqref{a1} and \eqref{basic-est}, we deduce
\begin{align}\label{k3}
&\frac{\dif}{\dif t}\left(t\left(\|\sqrt{\rho}u_t\|_{L^2}^2+\|\sqrt{\rho}w_t\|_{L^2}^2\right)\right) 
+t\bar\rho^{\alpha}\left(\bar{\mu}\|\D u_t\|_{L^2}^2+\bar{\eta}\|\D w_t\|_{L^2}^2\right) \nonumber\\
\leq~& \|\sqrt{\rho}u_t\|_{L^2}^2+\|\sqrt{\rho}w_t\|_{L^2}^2+Ct\left(\bar{\rho}^{\frac{3}{2}-2\alpha}+\bar{\rho}^{-\frac{3}{2}+(\frac{1}{2}-\alpha)\frac{6}{q}}\right)\left(\|\D u\|_{L^2}+\|\D w\|_{L^2}\right)\left(\|\sqrt{\rho}u_t \|_{L^2}^3+\|\sqrt{\rho}w_t\|_{L^2}^3\right) \nonumber\\
&+Ct\left(\bar{\rho}^{3-3\alpha}+\bar{\rho}^{1+\beta-2\alpha}+\bar{\rho}^{ -\alpha+\frac{6}{q}(\frac{1}{2}-\alpha)}\right)\left(\|\D u\|_{L^2}^2+\|\D w\|_{L^2}^2\right)\left(\|\sqrt{\rho}u_t \|_{L^2}^2+\|\sqrt{\rho}w_t\|_{L^2}^2\right) \nonumber\\
&+Ct\left(\bar{\rho}^{ -\alpha+2(1-\alpha)\frac{6}{q}}+\bar{\rho}^{-\alpha+\frac{6}{q}(\beta-\alpha)}+\bar{\rho}^{ -\alpha+\frac{6}{q}(\frac{1}{2}-\alpha)}+\bar{\rho}^{\beta-2+(\beta-\alpha)\frac{6}{q}}\right)\left(\|\nabla u\|_{L^2}^4+\|\nabla w\|_{L^2}^4\right),
\end{align}
due to $q\in (3,6)$ and \eqref{a1}. Thus, Gronwall's inequality yields
\begin{align}\label{k4}
&\sup_{0\le t\le T}t\left(\|\sqrt{\rho}u_t\|_{L^2}^2+\|\sqrt{\rho}w_t\|_{L^2}^2\right) +  \bar\rho^\alpha\int_{0}^{T} t\left(\|\nabla u_t\|^2_{L^2}+\|\nabla w_t\|^2_{L^2}\right)\dif t \nonumber\\
\leq~& C\int_0^T \bigg(t\left(\bar{\rho}^{ -\alpha+2(1-\alpha)\frac{6}{q}}+\bar{\rho}^{-\alpha+\frac{6}{q}(\beta-\alpha)}+\bar{\rho}^{ -\alpha+\frac{6}{q}(\frac{1}{2}-\alpha)}+\bar{\rho}^{\beta-2+(\beta-\alpha)\frac{6}{q}}\right)\left(\|\nabla u\|_{L^2}^4+\|\nabla w\|_{L^2}^4\right) \nonumber\\
&+\|\sqrt{\rho}u_t\|_{L^2}^2+\|\sqrt{\rho}w_t\|_{L^2}^2 \bigg) \dif t \nonumber\\
&\cdot\exp{\left\{\left(\bar{\rho}^{\frac{3}{2}-2\alpha}+\bar{\rho}^{-\frac{3}{2}+(\frac{1}{2}-\alpha)\frac{6}{q}}\right)\int_0^T \left(\|\D u\|_{L^2}+\|\D w\|_{L^2}\right)\left(\|\sqrt{\rho}u_t \|_{L^2}+\|\sqrt{\rho}w_t\|_{L^2}\right)\dif t\right\}} \nonumber\\
&\cdot\exp{\left\{ \left(\bar{\rho}^{3-3\alpha}+\bar{\rho}^{1+\beta-2\alpha}+\bar{\rho}^{ -\alpha+\frac{6}{q}(\frac{1}{2}-\alpha)}\right)\int_0^T\left(\|\D u\|_{L^2}^2+\|\D w\|_{L^2}^2\right)\dif t\right\}}.
\end{align}

Now we estimate the terms on the right hand side of \eqref{k4}. By \eqref{a1}, \eqref{basic-est} and \eqref{tdu}, we have
\begin{equation}\label{111}
\int_0^T t\left(\|\nabla u\|_{L^2}^4+\|\nabla w\|_{L^2}^4\right) \dif t\leq\sup_{t\in[0,T]}t(\|\D u \|_{L^2}^2+\|\D w\|_{L^2}^2)\int_0^T(\|\D u \|_{L^2}^2+\|\D w\|_{L^2}^2)\dif t\leq C\bar{\rho}^{2(1-\alpha)}.
\end{equation}
Next, H\"older's inequality yields
\begin{align}\label{113}
&\int_0^T \left(\|\D u\|_{L^2}+\|\D w\|_{L^2}\right)\left(\|\sqrt{\rho}u_t \|_{L^2}+\|\sqrt{\rho}w_t\|_{L^2}\right)\dif t \nonumber\\
\leq~&\left(\int_0^T(\|\D u \|_{L^2}^2+\|\D w\|_{L^2}^2)\dif t\right)^{\frac{1}{2}}\left(\int_0^T\left(\|\sqrt{\rho}u_t \|_{L^2}^2+\|\sqrt{\rho}w_t\|_{L^2}^2\right)\dif t\right)^{\frac{1}{2}}\leq C\bar\rho^{\frac{1}{2}}.
\end{align}
Inserting \eqref{111}--\eqref{113} into \eqref{k4}, one gets
\begin{equation}\label{trut1}
\sup_{0\le t\le T}t\left(\|\sqrt{\rho}u_t\|_{L^2}^2+\|\sqrt{\rho}w_t\|_{L^2}^2\right) +  \bar\rho^\alpha\int_{0}^{T} t\left(\|\nabla u_t\|^2_{L^2}+\|\nabla w_t\|^2_{L^2}\right)\dif t\leq C\bar{\rho}^\alpha\exp{\{C\bar{\rho}^{-A}\}}\leq C\bar{\rho}^\alpha.
\end{equation}
for some constant $A>0$. This completes the proof of \eqref{trut}.

On the other hand, multiplying \eqref{k3} by $t$, one has
\begin{align}
&\frac{\dif}{\dif t}\left(t^2\left(\|\sqrt{\rho}u_t\|_{L^2}^2+\|\sqrt{\rho}w_t\|_{L^2}^2\right)\right) 
+t^2\bar\rho^{\alpha}\left(\bar{\mu}\|\D u_t\|_{L^2}^2+\bar{\eta}\|\D w_t\|_{L^2}^2\right) \nonumber\\
\leq~& Ct(\|\sqrt{\rho}u_t\|_{L^2}^2+\|\sqrt{\rho}w_t\|_{L^2}^2)+Ct^2\left(\bar{\rho}^{\frac{3}{2}-2\alpha}+\bar{\rho}^{-\frac{3}{2}+(\frac{1}{2}-\alpha)\frac{6}{q}}\right)\left(\|\D u\|_{L^2}+\|\D w\|_{L^2}\right)\left(\|\sqrt{\rho}u_t \|_{L^2}^3+\|\sqrt{\rho}w_t\|_{L^2}^3\right) \nonumber\\
&+Ct^2\left(\bar{\rho}^{3-3\alpha}+\bar{\rho}^{1+\beta-2\alpha}+\bar{\rho}^{ -\alpha+\frac{6}{q}(\frac{1}{2}-\alpha)}\right)\left(\|\D u\|_{L^2}^2+\|\D w\|_{L^2}^2\right)\left(\|\sqrt{\rho}u_t \|_{L^2}^2+\|\sqrt{\rho}w_t\|_{L^2}^2\right) \nonumber\\
&+Ct^2\left(\bar{\rho}^{ -\alpha+2(1-\alpha)\frac{6}{q}}+\bar{\rho}^{-\alpha+\frac{6}{q}(\beta-\alpha)}+\bar{\rho}^{ -\alpha+\frac{6}{q}(\frac{1}{2}-\alpha)}+\bar{\rho}^{\beta-2+(\beta-\alpha)\frac{6}{q}}\right)\left(\|\nabla u\|_{L^2}^4+\|\nabla w\|_{L^2}^4\right).
\end{align}
Applying Gronwall’s inequality yields
\begin{align}
&\sup_{0\le t\le T}t^2\left(\|\sqrt{\rho}u_t\|_{L^2}^2+\|\sqrt{\rho}w_t\|_{L^2}^2\right) +  \bar\rho^\alpha\int_{0}^{T} t^2\left(\|\nabla u_t\|^2_{L^2}+\|\nabla w_t\|^2_{L^2}\right)\dif t \nonumber\\
\leq~& C\int_0^T \bigg(t^2\left(\bar{\rho}^{ -\alpha+2(1-\alpha)\frac{6}{q}}+\bar{\rho}^{-\alpha+\frac{6}{q}(\beta-\alpha)}+\bar{\rho}^{ -\alpha+\frac{6}{q}(\frac{1}{2}-\alpha)}+\bar{\rho}^{\beta-2+(\beta-\alpha)\frac{6}{q}}\right)\left(\|\nabla u\|_{L^2}^4+\|\nabla w\|_{L^2}^4\right) \nonumber\\
&+t(\|\sqrt{\rho}u_t\|_{L^2}^2+\|\sqrt{\rho}w_t\|_{L^2}^2)\bigg) \dif t \cdot\exp{\{\bar\rho^{-A}\}} \nonumber\\
\leq~&C\bar\rho\cdot\exp{\{\bar\rho^{-A}\}} \nonumber\\
\leq~&C\bar{\rho},
\end{align}
which completes the proof of \eqref{ttrut}.
\end{proof}

Finally, we are in a position to close the a priori assumption of  $\mathcal{E}_\rho$, the key observation is that $\|\nabla u\|_{L^1_tL^\infty_x}$ is uniformly bounded with respect to time $T$.
\begin{lemma}\label{L_1}
Under the conditions of Proposition \ref{pr}, there exists a positive constant $M_4$ such that
\begin{equation}
\mathcal{E}_\rho(T)\le  2 \mathcal{E}_\rho(0),
\end{equation}
provided $\bar{\rho}\geq M_4(\Omega,C_0,\bar\mu,\bar\xi,\bar\eta, \bar\lambda,\alpha,\beta,\|\D\rho_0\|_{L^q},\|u_0\|_{H^2},\|w_0\|_{H^2})$.
\end{lemma}
\begin{proof}
Applying the operator $\nabla$ to \eqref{ins}$_1$, we get  
	\begin{equation}\label{L11}
		\nabla \rho_t + u \cdot \nabla^2 \rho + \nabla u \cdot \nabla \rho =0.
	\end{equation}
Making the $L^2$-inner product of \eqref{L11} with $|\nabla \rho|^{q-2}\nabla \rho$ and then integrating by parts, we have 
\begin{equation}\label{dr}
\frac{1}{p}\frac{\dif}{\dif t} \|\nabla \rho\|_{L^q}^q  = - \int \nabla \rho \cdot \nabla u \cdot \nabla \rho |\nabla \rho|^{q-2} \dif x
\le C \|\nabla u\|_{L^\infty} \|\nabla \rho \|_{L^q}^q.
\end{equation}
It follows from Gronwall's inequality that
\begin{equation}\label{L12}
\|\D\rho\|_{L^q}^q\leq\|\D\rho_0\|_{L^q}^q\cdot\exp\Big\{C\int_0^T\|\D u\|_{L^\infty}\dif t\Big\}.
\end{equation}
Therefore, if we can prove the following estimate
\begin{equation}\label{Du1infty}
\int_0^T\|\D u\|_{L^\infty}\dif t\leq C\bar{\rho}^{-B},
\end{equation}
for some constant $B>0$, then by taking $M_4$ sufficiently large such that $C\bar{\rho}^{-B}\leq\ln 2$, we complete the proof of Lemma \ref{L_1}.

Adding \eqref{W2q} and \eqref{W2qw} together, we get
\begin{equation}\label{W2quw}
\|u\|_{W^{2,q}}+\|w\|_{W^{2,q}}\leq C\left(\bar\rho^{-\alpha}(\|\rho u_t\|_{L^q}+\|\rho w_t\|_{L^q}) +  \bar \rho^{(1-\alpha)\frac{5q-6}{q}} \|\nabla u\|_{L^2}^{\frac{5q-6}{q}}(\|\D u\|_{L^2}+\|\D w\|_{L^2})\right).
\end{equation}
It follows from \eqref{basic-est}, \eqref{Euw} and \eqref{W2quw} that
\begin{align}\label{kkk}
&\int_0^T\|\D u\|_{L^\infty}\dif t\leq C\int_0^T\|\D u\|_{W^{1,q}}\dif t \nonumber\\
\leq~& C\bar{\rho}^{-\alpha}\int_0^T\left(\|\rho u_t\|_{L^q}+\|\rho w_t\|_{L^q}\right)\dif t+C\bar{\rho}^{(1-\alpha)\frac{5q-6}{q}}\int_0^T\|\nabla u\|_{L^2}^{\frac{5q-6}{q}}(\|\D u\|_{L^2}+\|\D w\|_{L^2})\dif t \nonumber\\
\leq~& C\bar{\rho}^{-\alpha}\int_0^T\left(\|\rho u_t\|_{L^q}+\|\rho w_t\|_{L^q}\right)\dif t+C\bar{\rho}^{(1-\alpha)\frac{5q-6}{q}}\int_0^T\|\nabla u\|_{L^2}^2\dif t \nonumber\\
\leq~& C\bar{\rho}^{-\alpha}\int_0^T\left(\|\rho u_t\|_{L^q}+\|\rho w_t\|_{L^q}\right)\dif t+C\bar{\rho}^{\frac{6(q-1)}{q}(1-\alpha)}.
\end{align}
Using Gagliardo-Nirenberg inequality, \eqref{trut} and \eqref{ttrut}, we have 
\begin{align}\label{long}
\int_{0}^{T}\left(\|\rho u_t\|_{L^q}+\|\rho w_t\|_{L^q}\right)\dif t\leq~& C\bar\rho^{\frac{5q-6}{4q}}\int_{0}^{T}\left(\|\sqrt{\rho}u_t\|_{L^2}^{\frac{6-q}{2q}} \|\nabla u_t\|_{L^2}^{\frac{3(q-2)}{2q}}+\|\sqrt{\rho}w_t\|_{L^2}^{\frac{6-q}{2q}} \|\nabla w_t\|_{L^2}^{\frac{3(q-2)}{2q}}\right)\dif t \nonumber\\
\leq~& C \bar\rho^{\frac{5q-6}{4q}} \left(\sup_{0\le t\le \min\{1,T\}} t \|\sqrt{\rho}u_t\|_{L^2}^2 \dif t\right)^{\frac{6-q}{4q}} \nonumber\\
&\cdot \left( \int_{0}^{\min\{1,T\}} t \|\nabla u_t\|_{L^2}^2 \dif t \right)^{\frac{3(q-2)}{4q}}\left( \int_{0}^{\min\{1,T\}} t^{-\frac{2q}{q+6}}  \dif t \right)^{\frac{q+6}{4q}} \nonumber\\
&+ C \bar\rho^{\frac{5q-6}{4q}} \left(\sup_{\min\{1,T\}\le t\le T} t^2 \|\sqrt{\rho}u_t\|_{L^2}^2 \dif t\right)^{\frac{6-q}{4q}} \nonumber\\
&\cdot\left( \int_{\min\{1,T\}}^{T} t^2 \|\nabla u_t\|_{L^2}^2 \dif t \right)^{\frac{3(q-2)}{4q}} \left( \int_{\min\{1,T\}}^{T} t^{-\frac{4q}{q+6}}  \dif t \right)^{\frac{q+6}{4q}} \nonumber\\
&+C \bar\rho^{\frac{5q-6}{4q}} \left(\sup_{0\le t\le \min\{1,T\}} t \|\sqrt{\rho}w_t\|_{L^2}^2 \dif t\right)^{\frac{6-q}{4q}} \nonumber\\
&\cdot \left( \int_{0}^{\min\{1,T\}} t \|\nabla w_t\|_{L^2}^2 \dif t \right)^{\frac{3(q-2)}{4q}}\left( \int_{0}^{\min\{1,T\}} t^{-\frac{2q}{q+6}}  \dif t \right)^{\frac{q+6}{4q}} \nonumber\\
&+ C \bar\rho^{\frac{5q-6}{4q}} \left(\sup_{\min\{1,T\}\le t\le T} t^2 \|\sqrt{\rho}w_t\|_{L^2}^2 \dif t\right)^{\frac{6-q}{4q}} \nonumber\\
&\cdot\left( \int_{\min\{1,T\}}^{T} t^2 \|\nabla w_t\|_{L^2}^2 \dif t \right)^{\frac{3(q-2)}{4q}} \left( \int_{\min\{1,T\}}^{T} t^{-\frac{4q}{q+6}}  \dif t \right)^{\frac{q+6}{4q}} \nonumber\\
\leq~& C \bar\rho^{\frac{5q-6}{4q} + \alpha \frac{6-q}{4q} } + C \bar\rho^{\frac{5q-6}{4q} + \frac{6-q}{4q} + (1-\alpha) \frac{3(q-2)}{4q}} \nonumber\\
\leq~& C \bar\rho^{\frac{5q-6}{4q} + \alpha \frac{6-q}{4q} }.
\end{align}

Combining \eqref{kkk} and \eqref{long}, we obtain
\begin{equation}
\int_0^T \|\nabla u\|_{L^\infty} \dif t \leq C \bar\rho^{\frac{5q-6}{4q} + \alpha \frac{6-q}{4q} -\alpha } +  C \bar{\rho}^{\frac{6(q-1)}{q}(1-\alpha)}\leq  C \bar\rho^{-D},
\end{equation}
where
\begin{equation*}
D=\min{\left\{\frac{5q-6}{4q}(\alpha-1), \frac{6(q-1)}{q}(\alpha-1)\right\}}.
\end{equation*}
Thus we complete the proof.
\end{proof}

At last, combining Lemma \ref{L_2} and \ref{L_1}, we have proved Proposition \ref{pr}. With the help of Lemma \ref{uniform.rho}--\ref{L_1}, Theorem \ref{global} can be established through a standard bootstrap argument (see \cite{Huang-Li-Zhang-arxiv}). We omit it here.

\bigskip

\noindent {\bf Acknowledgments}\\
This work is supported by Science Foundation of Zhejiang Sci-Tech University (ZSTU) under Grant No. 25062122-Y. The authors would like to express their sincere thanks to Professor Yachun Li for the helpful suggestions and discussions.

\bigskip 
 
\noindent{\bf Data Availability Statements}\\ 
Data sharing not applicable to this article as no datasets were generated or analysed during the current study.

\bigskip

\noindent{\bf Conflict of interests}\\
The authors declare that they have no competing interests.

\bigskip

\noindent{\bf Authors' contributions}\\
The authors have made the same contribution. All authors read and approved the final manuscript.

\end{document}